\documentclass[11pt]{amsart}
\usepackage{geometry}                
\usepackage{graphicx}
\usepackage{amssymb}
\usepackage{epstopdf}
\theoremstyle{plain}

\newtheorem{thm}{Theorem}[section]
\newtheorem{cor}[thm]{Corollary}
\newtheorem{prop}[thm]{Proposition}
\newtheorem{lem}[thm]{Lemma}

\theoremstyle{definition}

\newtheorem{rmk}[thm]{Remark}

\newcommand {\PP}{\mathbb{P}}

\title{Ulrich bundles on ruled surfaces}
\author{Marian Aprodu, Laura Costa, Rosa Maria Mir\'o-Roig}
\date{\today}                                           
\address{Facultatea de Matematic\u a \c si Informatic\u a, Universitatea din Bucure\c sti, Str. Academiei 14, 010014 Bucure\c sti, ROMANIA \& Institutul de Matematic\u a "Simion Stoilow" al Academiei Rom\^ane, Calea Grivi\c tei 21, Sector 1, 010702 Bucure\c sti, ROMANIA}
\email{marian.aprodu@fmi.unibuc.ro \& marian.aprodu@imar.ro}
\address{Facultat de Matem\`atiques i Inform\`{a}tica,
Departament de Matem\`{a}tiques i Inform\`{a}tica, Gran Via de les Corts Catalanes
585, 08007 Barcelona, SPAIN } \email{costa@ub.edu}
\address{Facultat de Matem\`atiques i Inform\`{a}tica,
Departament de Matem\`{a}tiques i Inform\`{a}tica, Gran Via de les Corts Catalanes
585, 08007 Barcelona, SPAIN } \email{miro@ub.edu}

\thanks{Marian Aprodu thanks the Max Planck Institute for Mathematics Bonn for hospitality during the preparation of this work The second and third authors were partially supported by MTM2013-45075-P}
\begin{document}
\maketitle
\begin{abstract} In this short note, we study the existence problem for Ulrich bundles on ruled surfaces, focusing our attention on the smallest possible rank. We show that existence of Ulrich line bundles occurs if and only if the coefficient $\alpha$ of the minimal section in the numerical class of the polarization equals one. For other polarizations, we prove the existence of rank two Ulrich bundles.
\end{abstract}

\section{Introduction}

Let $(X,H)$ be a smooth complex projective variety of dimension $n$. An \emph{Ulrich bundle} on $X$ is a vector bundle $E$ which satisfies
$H^i (X, E(-iH))=0 \text{ for all } i>0$
and
$H^j (X, E(-(j+1)H)=0 \text { for all } j<n.$  Its pushforward to $\PP^m$ via the inclusion of $X$ given by the ample line bundle $H$ has a linear resolution or, equivalently, for any linear projection $X\to \mathbb P^n$, its pushforward is trivial (see \cite{ES03}, Proposition 2.1).

Ulrich bundles have made a first appearance in commutative algebra, being associated to maximal Cohen-Macaulay graded modules with  maximal number of generators, \cite{Ul84}. In algebraic geometry, their importance was underlined in \cite{Bea00} and \cite{ES03} where a relation between their existence and the representations of Cayley--Chow forms was found. They have also important applications in liaison theory, singularity theory, moduli spaces and Boij-S\"{o}derberg theory. For example, if $X$ admits an Ulrich bundle, then its cone of cohomology tables coincide with that of           $\PP^n$ \cite{ES11}.
In view of their importance, Eisenbud, Schreyer and Weyman formulated the problem of determining which varieties $X$ admits an Ulrich bundle and, if so, what is the smallest rank of an Ulrich bundle that they carry. So far few examples are known and we refer the reader to \cite{CH12}, \cite{CKM12}, \cite{CKM13}, \cite{CMR}, \cite{CMP},
\cite{ES03}, \cite{KM}, \cite{Mi}, \cite{MP} and \cite{MP2}
for references and further results.
 In this context the existence of undecomposable  rank 2 Ulrich bundles on a surfaces is a challenging problem and
 a particular attention is focused on the existence of special rank two Ulrich bundles, i.e. rank 2 0-regular vector bundles $E$ on $X$ such that $c_1(E)=K_X+3H$. The existence of a special Ulrich bundle implies that the associated Cayley--Chow form is represented as a linear pfaffian \cite{ES03}. Recently, this problem has been solved for general K3 surfaces \cite{AFO12}, \cite{CKM12}, for del Pezzo surfaces in \cite{MP2}, for ACM rational surfaces in \cite{MP} and for surfaces with $q=p_g=0$, \cite{Bea16}. We will focus here on surfaces with Kodaira dimension $-\infty$.
In this note, we study the existence problem for Ulrich bundles on ruled surfaces, which is a relevant case thanks to \cite{Kim16}. 
\medskip

The structure of the paper is the following. In Section \ref{sec:rank1} we analyze the existence of Ulrich line bundles. We show that existence occurs if and only if the coefficient $\alpha$ of the minimal section in the numerical class of the polarization equals one, Theorem \ref{thm:rank1}. In Section \ref{sec:rank2} we prove existence of rank two special Ulrich bundles with respect to arbitrary polarizations (Proposition \ref{prop:rank2alpha1}, Proposition \ref{prop:rank2alpha2} and  Theorem \ref{thm:rank2}).  The rank--two bundles that we construct are strictly semistable for $\alpha=1$ and stable for $\alpha\ge 2$, Remark \ref{rmk:stability}. For basic properties of Ulrich bundles, we refer to the literature included here, especially \cite{ES03}.

\medskip

\noindent
{\bf Notation.} $\pi:X\to C$ is a geometrically ruled surface of invariant $e>0$ over the genus--$g$ curve $C$. Fix $p\in C$ and denote $F$ the fibre of $\pi$ over $p$. Denote by $C_0$ the section of self--intersection $-e$. Write $\mathcal O_X(K_X)=\mathcal O_X(-2C_0+(2g-2-e)F)\otimes \pi^*\mathfrak k$, with $\mathfrak k\in \mathrm{Pic}^0(C)$. We shall work with a polarization $H_{\alpha}$ on $X$ of numerical class $H\equiv \alpha C_0+\beta F$ with $\alpha \ge 1$ and $\beta>\alpha e$.

\section{Ulrich line bundles}
\label{sec:rank1}

In this section, we analyze the existence of Ulrich line bundles. We note first that, if they exists, Ulrich line bundles come in pairs.
Indeed, if $\mathcal O_X(L)$ is an Ulrich line bundle on $X$ with respect to the very ample line bundle $\mathcal O_X(H_{\alpha})$ then
\[ \mathcal O_X(K_X+3H_{\alpha}-L)\]
is also an Ulrich line bundle on $X$. In fact, $\mathcal O_X(K_X+3H_{\alpha}-L)$ is an Ulrich line bundle if and only if for any $i \geq 0$ and $-2 \leq t \leq -1$, $H^i(\mathcal O_X(K_X+3H_{\alpha}-L-tH_{\alpha}))=0$. Hence the claim follows from the fact that since  $\mathcal O_X(L)$ is an Ulrich line bundle,  for any $i \geq 0$ and $-2 \leq t \leq -1$, $H^i(\mathcal O_X(L-tH_{\alpha}))=0$ and from the fact that from Serre's duality we have
\[H^i(\mathcal O_X(K_X+3H_{\alpha}-L-tH_{\alpha}))=H^{2-i}(\mathcal O_X(L+(t-3)H_{\alpha}))=0. \]

The main result of this section is the following.

\begin{thm}
\label{thm:rank1}
Let $X$ be  a geometrically ruled surface over a smooth curve $C$ of genus $g$ and with $e>0$. Let $H_{\alpha}=\alpha C_0+\beta F$ be any very ample divisor on $X$. Then, there are Ulrich line bundles with respect to $H_{\alpha}$ if and only if $\alpha=1$. In this case, there exists exactly two types of Ulrich line bundles on $X$, namely
\[ \mathcal O_X(L_1):= \mathcal O_X((2\beta+g-1-e)F)\otimes  \pi^*\mathcal{L}_1\]
and
\[
\mathcal O_X(L_2):= \mathcal O_X(K_X+3H_{\alpha}-L_1)=\mathcal O_X(C_0+(\beta+g-1)F)\otimes \pi^*\mathcal L_2
\]
with $\mathcal L_1,\mathcal L_2\in\mathrm{Pic}^0(C)$ general. 
\end{thm}
\begin{proof}
 Let $\mathcal O_X(L_1)$ be an Ulrich line bundle on $X$. Then its pair $\mathcal O_X(L_2):=\mathcal O_X(K_X+3H_{\alpha}-L_1)$ is also Ulrich. In particular, this means that there exists two divisors $D_i=L_i-L\equiv a_iC_0+b_iF$, i.e. $\mathcal O_X(D_i)=\mathcal O_X(a_iC_0+b_iF)\otimes \pi^*\mathcal{L}_i$ for $a_i,b_i\in \mathbb Z$ and $\mathcal L_i\in\mathrm{Pic}^0(C)$,  such that $\chi (\mathcal O_X(D_i))=0$, $i\in\{1,2\}$.
 By Riemann--Roch's theorem this means that
\[ 0=\chi(\mathcal O_X(D_i))=1-g-\frac{D_i\cdot K_X}{2}+\frac{D_i^2}{2}=(a_i+1)(b_i+1-g-\frac{a_ie}{2}). \]
Therefore $a_i=-1$ or $b_i=g-1+\frac{a_ie}{2}$ for $i\in\{1,2\}$. Note that $a_1+a_2=\alpha-2$ and $b_1+b_2=\beta+2g-2-e$.

\medskip

We claim that the equations $b_i=g-1+\frac{a_ie}{2}$ cannot be satisfied simultaneously. Indeed,
if the two conditions are satisfied in the same time, then, by adding them up and replacing $b_1+b_2$ and $a_1+a_2$ we obtain $\beta=\alpha e$ which contradicts the very ampleness of $H_{\alpha}$.

On the other hand, since $a_1+a_2=\alpha -2$, $a_1$ and $a_2$ cannot be simultaneously $-1$, for that would imply $\alpha=0$, contradiction.

Hence either $a_1=-1$ and $b_2=g-1+\frac{a_2e}{2}$ or $a_2=-1$ and $b_1=g-1+\frac{a_2e}{2}$.

Assume that $a_1=-1$ and $b_2=g-1+\frac{a_2e}{2}$ (the second case is similar, due to the symmetry of the situation). Then $a_2=\alpha-1$ and $b_2=g-1+\frac{(\alpha-1)e}{2}$, hence
\[
\mathcal O_X(L_2-H_{\alpha})=\mathcal O_X\left((\alpha-1)C_0+\left(g-1+\frac{e(\alpha-1)}{2}\right)F\right)\otimes \pi^*\mathcal L_2.
\]

Since $e\ge 1$, $\alpha\ge 1$, and $h^0(\mathcal O_X(L_2-H_{\alpha}))=0$ it follows that $\alpha =1$. Indeed, if $\alpha \ge 2$, the degree of $\mathcal O_C\left(\left(g-1+\frac{e(\alpha-1)}{2}\right)p\right)\otimes \mathcal L_2$ is at least $g$ which means that $\mathcal O_X\left(\left(g-1+\frac{e(\alpha-1)}{2}\right)F\right)\otimes \pi^*\mathcal L_2$ has sections, implying also that the line bundle
\[
\mathcal O_X\left((\alpha-1)C_0+\left(g-1+\frac{e(\alpha-1)}{2}\right)F\right)\otimes \pi^*\mathcal L_2
\]
has sections. In conclusion, $\alpha =1$, which implies $\mathcal O_X(L_1)= \mathcal O_X((2\beta+g-1-e)F)\otimes  \pi^*\mathcal{L}_1$.

 \medskip

Conversely, assume that $\alpha=1$ and we prove that
\[
\mathcal O_X(L_1):= \mathcal O_X((2\beta+g-1-e)F)\otimes  \pi^*\mathcal{L}_1,
\]
 with $\mathcal L_1\in\mathrm{Pic}^0(C)$ general, is an Ulrich bundle. To this end, observe that
\[
H_{\alpha}\cdot L_1=\frac{H_{\alpha}\cdot(K_X+3H_{\alpha})}{2}
\]
and that by Riemann--Roch's Theorem
\[ \chi (\mathcal O_X(L_1-H_{\alpha}))= (1-g)-\frac{1}{2}(L_1-H_{\alpha})\cdot K_X+\frac{1}{2}(L_1-H_{\alpha})^2=0. \]
These two facts also imply that $\chi (\mathcal O_X(L_1-2H_{\alpha}))=0$ because
\[
\chi (\mathcal O_X(L_1-2H_{\alpha}))= \chi (\mathcal O_X(L_1-H_{\alpha}))-H_{\alpha}\cdot L_1+\frac{H_{\alpha}\cdot (K_X+3H_{\alpha})}{2}.
\]
On the other hand, $\mathcal O_X(L_1)$ is an Ulrich line bundle if and only if for any $i \geq 0$ and any $-2 \leq t \leq -1$, we have $H^i(\mathcal O_X (L_1-tH_{\alpha}))=0$. Notice that the coefficient of $C_0$ in $L_1-H_{\alpha}$ equals $-1$, hence $h^0(\mathcal O_X(L_1-H_{\alpha}))=0$ which also implies that $h^0(\mathcal O_X(L_1-2H_{\alpha}))=0$. The coefficient of $C_0$ in $H_{\alpha}-L_1+K_X$ equals also $-1$ and hence  $h^2(\mathcal O_X(L_1-H_{\alpha}))= h^0(\mathcal O_X(H_{\alpha}-L_1+K_X))=0$. These vanishing together with the fact that $\chi (\mathcal O_X(L_1-H_{\alpha}))=0$ give us $h^1(\mathcal O_X(L_1-H_{\alpha}))=0$.  Finally, $\mathcal O_X(2H_{\alpha}-L_1+K_X)=\mathcal O_X((g-1)F)\otimes\mathcal \pi^*\mathcal L_2$ with $\mathcal L_2\in\mathrm{Pic}^0(C)$ general and hence
\[
h^2(\mathcal O_X(L_1-2H_{\alpha}))= h^0(\mathcal O_X(2H_{\alpha}-L_1+K_X))= h^0(\mathcal O_C((g-1)p)\otimes \mathcal L_2)=0.
\]

Since  $\chi (\mathcal O_X(L_1-2H_{\alpha}))=0$ we also have $h^1(\mathcal O_X(L_1-2H_{\alpha}))= 0$. Putting altogether we get that $\mathcal O_X(L_1)$ is an Ulrich line bundle and hence the same holds for $\mathcal O_X(L_2)$.
\end{proof}

\section{Rank-two Ulrich bundles}
\label{sec:rank2}

The goal of this section is to prove the existence of special rank 2 Ulrich bundles with respect to any very ample divisor $H_{\alpha}=\alpha C_0+ \beta F$ on $X$. To this end, we will start analyzing separately the cases $\alpha=1$ and $\alpha=2$.

\begin{prop}
\label{prop:rank2alpha1}
Let $X$ be  a geometrically ruled surface over a smooth curve $C$ of genus $g$ and with $e>0$.
Let $H=C_0+\beta F$ be a very ample divisor on $X$.
Then there exists a family of dimension $2 \beta -e+4g-3$  of indecomposable rank-two simple strictly semi-stable Ulrich bundles on $X$.
\end{prop}

\proof
It follows from Theorem \ref{thm:rank1} that \[ \mathcal O_X(L_1):= \mathcal O_X((2\beta-1-e)F)\otimes \pi^*\mathcal L_1 \quad \mbox{and} \quad
 \mathcal O_X(L_2):= \mathcal O_X(K_X+3H-L_1)\otimes \pi^*\mathcal L_2 \]
 are Ulrich line bundles on $X$ for general $\mathcal L_1 \in Pic^0(C)$ and $\mathcal L_2=\mathfrak k\otimes \mathcal L_1^{-1}$.
Moreover,  $h^0(\mathcal O_X(L_2-L_1))=h^2(\mathcal O_X(L_2-L_1))=0$ and hence
\[ \dim \mathrm{Ext}^1(\mathcal O_X(L_1),\mathcal O_X(L_2))= h^1(\mathcal O_X(L_2-L_1))=- \chi (\mathcal O_X(L_2-L_1))= 2 \beta -e+2(g-1) >0. \]
Therefore we can consider rank two vector bundles $E$ on $X$ given by a non trivial extension of the following type
\[
0\to \mathcal O_X( C_0+(\beta-1)F)\otimes \pi^*\mathcal L_2 \to E \to \mathcal O_X((2 \beta -1-e)F)\otimes \pi^*\mathcal L_1  \to 0
\]
Since $\mathcal O_X(L_1)\otimes \pi^*\mathcal L_1$ and  $ \mathcal O_X(L_2)\otimes \pi^*\mathcal L_2$ are Ulrich line bundles, $E$ is also an Ulrich  bundle. Moreover it is simple and in particular indecomposable due to the fact that there are no nonzero morphisms from $L_i$ to $L_j$ for $i\ne j$. Finally notice that since Ulrich bundles are semi--stable and the slope of $L_2$ and $E$ coincide, it is strictly semi-stable.

Finally, since $h^0(E(-C_0-(\beta-1)F))=1$, we have a  family of dimension $2 \beta -e+4g-3 $  of indecomposable rank-two simple strictly semi-stable Ulrich bundles on~$X$.
\endproof

Before we  pass  to the case where the coefficient of $C_0$ in the polarization $H_{\alpha}$ is equal or greater than two, we will prove the following result that can be stated in a more general setting:

\begin{lem}
\label{lem:initialized}
Let $S$ be a surface, $H$ be a very ample divisor on $S$, $E$ a rank-two vector bundle on $S$ with $\mathrm{det}(E)=\mathcal O_S(K_S+3H)$ and $c_2(E)=\frac{1}{2}H\cdot (5H+3K_S)+2\chi(\mathcal O_S)$. Then $E$ is Ulrich if and only if $E$ is initialized i.e. $H^0(E(-H))=0$.
\end{lem}

\proof
If $E$ is Ulrich, then by definition $H^0(E(-H))=0$ and hence we only have to prove the converse. For surfaces, Ulrichness of $E$ reduces to
\[
h^0(E(-H))=h^1(E(-H))=h^1(E(-2H))=h^2(E(-2H))=0.
\]

On the other hand, $E^*\cong E(-K_S-3H)$ and hence Serre's duality gives $h^2(E(-2H))=h^0(E^*(K_S+2H))=h^0(E(-H))$. Note also that the vanishing of $h^0(E(-H))$ implies the vanishing of $h^0(E(-2H))$ ($H^0(E(-2H))\subset H^0(E(-H))$) and of $h^2(E(-H))$ by Serre duality.

The vanishing of $h^1(E(-H))$ and of $h^1(E(-2H))$ follows if we prove the vanishing of $\chi(E(-H))$ and of $\chi (E(-2H))$ respectively.

Notice that  $\mathrm{det}(E(-H))=K_S+H$ and $\mathrm{det}(E(-2H))=K_S-H$. On the other hand, $c_2(E(-H))=c_2(E)-c_1(E)\cdot H+H^2$ and $c_2(E(-2H))=c_2(E)-2c_1(E)\cdot L+4H^2$. Thus, from Riemann-Roch
\[
\chi(E(-H))=2\chi(\mathcal O_S)+\frac{1}{2}c_1(E(-H))\cdot (c_1(E(-H))-K_S)-c_2(E(-H))=0
\]
and similarly $\chi(E(-2H))=0$.
\endproof

Let us now consider the case where the coefficient of $C_0$ in the polarization $H_{\alpha}$ is equal to two:

\begin{prop}
\label{prop:rank2alpha2}
Let $X$ be  a geometrically ruled surface over a smooth curve $C$ of genus $g$ and with $e>0$. Let $H=2 C_0+\beta F$ be any very ample divisor on $X$. Then, there are special Ulrich bundles with respect to $H$ given by extensions
\[
0\to \mathcal O_X(2C_0+(\beta+g-1)F)\otimes \pi^*\mathcal L_1\to E
\to \mathcal I_Z(2C_0+(2\beta+g-1-e)F)\otimes\pi^*\mathcal L_2\to 0
\]
where $\mathcal L_1\in \mathrm{Pic}^0(C)$ is a general line bundle, $\mathcal L_2=\mathfrak k\otimes \mathcal L_1^{-1}$, and $Z$ a general zero--dimensional subscheme of $X$ with $\ell(Z)=\beta-e$.
\end{prop}

\proof
The bundles $E$ given by extensions as above fall under the statement of Lemma \ref{lem:initialized}: $c_1(E)\equiv K_X+3H$, $c_2(E)=4(g-1)+7\beta-7e$ and the line bundle $\mathcal L_2$ is determined by the condition $\mathrm{det}(E)=\mathcal O_X(K_X+3H)$. Hence, they are Ulrich if and only if they are initialized.

We have the short exact sequence
\[
0\to H^0(\mathcal O_X((g-1)F)\otimes \pi^*\mathcal L_1)\to H^0(E(-H))
\to H^0(\mathcal I_Z((\beta+g-1-e)F)\otimes\pi^*\mathcal L_2).
\]

For $\mathcal L_1$ general, $H^0(\mathcal O_X((g-1)F)\otimes \pi^*\mathcal L_1)=H^0(\mathcal O_C((g-1)p)\otimes \mathcal L_1)=0$, as $\mathcal O_C((g-1)p)\otimes \mathcal L_1$ is a general bundle in $\mathrm{Pic}^{g-1}(C)$. For $Z$ general $\pi_*\mathcal I_Z=\mathcal O_C(-(p_1+\ldots+p_{\beta-e}))$ with $p_i\in C$ distinct points and, since $\mathcal L_2=\mathfrak k\otimes \mathcal L_1^{-1}$ is also general, $\pi_*(\mathcal I_Z((\beta+g-1-e)F)\otimes\pi^*\mathcal L_2)$ is a general line bundle of degree $g-1$ and hence has no sections. It follows that $E$ is initialized.
\endproof

Finally, we consider the remaining cases.

\begin{thm}
\label{thm:rank2}
Let $X$ be  a geometrically ruled surface over a smooth curve $C$ of genus $g$ and with $e>0$. Let $H_\alpha=\alpha C_0+\beta F$ with $\alpha\ge 3$ be any very ample divisor on $X$ with
\begin{equation}
\label{cond}
2\beta>\mathrm{max}\{(\alpha-3)(g-1)+e\alpha,(g-1)+e\alpha\}.
\end{equation}
 Then, there are special Ulrich bundles with respect to $H_\alpha$ given by extensions
\[
0\to \mathcal O_X(\alpha C_0+(\beta+g-1)F)\otimes \pi^*\mathcal L_1\to E
\to \mathcal I_Z((2\alpha-2)C_0+(2\beta+g-1-e)F)\otimes\pi^*\mathcal L_2\to 0
\]
with $\mathcal L_1\in \mathrm{Pic}^0(C)$ a general line bundle, $\mathcal L_2=\mathfrak k\otimes \mathcal L_1^{-1}$ and $Z$ a general zero--dimensional subscheme of $X$ with $\ell(Z)=(\alpha-1)\left(\beta-\frac{e\alpha}{2}\right)$.
\end{thm}

\proof
We proceed in several steps.

\medskip

\noindent
{\em Step1.}
We prove first that there are vector bundles in the extensions like in the statement. To this end, we need to verify that $Z$ verifies the Cayley--Bacharach property with respect to the linear system
\[
|\mathcal O_X(K_X+(2\alpha-2)C_0+(2\beta+g-1-e)F-\alpha C_0-(\beta+g-1)F)\otimes\pi^*\mathcal L_2\otimes\pi^*\mathcal L_1^{-1})|=
\]
\[
=|\mathcal O_X(K_X+(\alpha-2)C_0+(\beta-e)F)\otimes\pi^*(\mathcal L_2\otimes\mathcal L_1^{-1}))|.
\]

Since $\beta>\alpha e>(\alpha-1)e$ and $\alpha\ge 3$, the class $(\alpha-2)C_0+(\beta-e)F$ is very ample. Hence, by the Kodaira vanishing 
\[
h^0(\mathcal O_X(K_X+(\alpha-2)C_0+(\beta-e)F)\otimes\pi^*(\mathcal L_2\otimes\mathcal L_1^{-1}))=
\]
\[
=\chi(\mathcal O_X(K_X+(\alpha-2)C_0+(\beta-e)F)\otimes\pi^*(\mathcal L_2\otimes\mathcal L_1^{-1}))=
\]
\[
=(\alpha-3)\left(g-1+\beta-\frac{e\alpha}{2}\right).
\]

The hypothesis (\ref{cond}), i.e. the condition $2\beta>(\alpha-3)(g-1)+e\alpha$ implies
\[
h^0(\mathcal O_X(K_X+(\alpha-2)C_0+(\beta-e)F)\otimes\pi^*(\mathcal L_2\otimes\mathcal L_1^{-1}))\le \ell(Z)-1,
\]
and it follows that for a general $Z$ and any $x\in\mathrm{supp}(Z)$, we have
\[
h^0(\mathcal I_{Z\setminus \{x\}}(K_X+(\alpha-2)C_0+(\beta-e)F)\otimes\pi^*(\mathcal L_2\otimes\mathcal L_1^{-1}))=0.
\]
Hence $Z$ trivially satisfies the Cayley--Bacharach property with respect to the linear system
\[
|\mathcal O_X(K_X+(\alpha-2)C_0+(\beta-e)F)\otimes\pi^*(\mathcal L_2\otimes\mathcal L_1^{-1}))|
\]
and therefore,  in any extension as in the statement there are rank-two vector bundles. It is clear that the determinant of such a bundle equals $\mathcal O_X(K_X+3H_\alpha)$. Moreover, a direct computation shows that the second Chern class is
\[
-\frac{\alpha(5\alpha-3)}{2}e+\beta(5\alpha -3)+(3\alpha-2)(g-1)=
\frac{1}{2}H_\alpha\cdot (5H_{\alpha}+3K_X)+2\chi(\mathcal O_X).
\]
So,  it is Ulrich if and only if it is initialized.

\medskip

\noindent
{\em Step 2.}
We prove that the rank two vector bundles we have constructed  are initialized.  We consider the  short exact sequence
\[
0\to H^0(\mathcal O_X((g-1)F)\otimes \pi^*\mathcal L_1)\to H^0(E(-H_\alpha))
\to H^0(\mathcal I_Z((\alpha-2)C_0+(\beta+g-1-e)F)\otimes\pi^*\mathcal L_2).
\]

The genericity of $\mathcal L_1$ ensures that $H^0(\mathcal O_X((g-1)F)\otimes \pi^*\mathcal L_1)=0$. The hypothesis $\beta>(g-1)+e\alpha$ implies that the class $\alpha C_0+(\beta-g+1)F$ is very ample. Thus,   Kodaira vanishing implies
\[
h^0(\mathcal O_X((\alpha-2)C_0+(\beta+g-1-e)F)\otimes\pi^*\mathcal L_2)=
\chi(\mathcal O_X((\alpha-2)C_0+(\beta+g-1-e)F)\otimes\pi^*\mathcal L_2)=
\]
\[
=-\frac{\alpha(\alpha-1)e}{2}+(\alpha-1)\beta=\ell(Z).
\]
Therefore, for a general $Z$ we have
\[
h^0(\mathcal I_Z((\alpha-2)C_0+(\beta+g-1-e)F)\otimes\pi^*\mathcal L_2)=0
\]
which finishes the proof.
\endproof

\begin{rmk}
The condition (\ref{cond}) is a very mild one. Besides, if $g=0$ or $g=1$, it is automatically satisfied.
\end{rmk}

\begin{rmk}
A version of Theorem \ref{thm:rank2}, with an almost verbatim proof, yields to existence of rank-two Ulrich bundles on Veronese surfaces, which is the content of  \cite{ES03} Proposition 5.9.
\end{rmk}

\begin{rmk}
\label{rmk:stability}
While the rank two vector bundles of Proposition \ref{prop:rank2alpha1} are strictly semistable, the rank two bundles from Proposition \ref{prop:rank2alpha2} and Theorem \ref{thm:rank2} are stable with respect to the given polarizations. Indeed, they are semistable and, if they were strictly semistable the the maximal destabilazing sequences would provide us with Ulrich line bundles, \cite[Theorem 2.9 (b)]{CH12}, which cannot exist, Theorem \ref{thm:rank1}.
\end{rmk}

Theorem \ref{thm:rank2} and \cite{Kim16} immediately yield the following

\begin{cor}
\label{cor:blowup}
Let $X\stackrel{\sigma}{\to}\mathbb P^2$ be the blowup of $\mathbb P^2$ at $n\ge 1$ general points, denote $E_1,\ldots E_n$ the exceptional divisors and let $H$ be an ample bundle on $X$ of type $\sigma^*\mathcal O_{\mathbb P^2}(k)-\sum E_i$. Then there exists a rank-2 Ulrich bundle on $X$ with respect to $H$.
\end{cor}

\end{document}